\documentclass[11pt]{amsart}

\usepackage[T1]{fontenc}
\usepackage[utf8]{inputenc}
\usepackage{lmodern}
\usepackage{amsmath,amssymb,amsthm,mathtools}
\usepackage{microtype}
\usepackage[margin=1in]{geometry}
\usepackage{enumitem}
\usepackage[colorlinks=true,linkcolor=blue,citecolor=blue,urlcolor=blue]{hyperref}

\allowdisplaybreaks
\numberwithin{equation}{section}

\newtheorem{theorem}{Theorem}[section]
\newtheorem{proposition}[theorem]{Proposition}
\newtheorem{lemma}[theorem]{Lemma}
\newtheorem{corollary}[theorem]{Corollary}
\theoremstyle{remark}
\newtheorem{remark}[theorem]{Remark}

\DeclareMathOperator{\lcm}{lcm}
\DeclareMathOperator{\Var}{Var}

\newcommand{\PP}{\mathbb P}
\newcommand{\EE}{\mathbb E}
\newcommand{\one}{\mathbf 1}
\newcommand{\vp}{\nu}

\newcommand{\von}{\Lambda}
\newcommand{\ZZ}{\mathbb Z}
\newcommand{\II}{\mathcal I}
\newcommand{\normal}{\mathcal N}

\newcommand{\fracp}[1]{\left\{#1\right\}}

\title[Erd\H{o}s Problem 684 at density one]{Erd\H{o}s Problem 684 at Density One: Small-prime Parts of Binomial Coefficients and Gaussian Fluctuations}

\author[E. Li]{Eric Li}
\dedicatory{\normalfont\normalsize Trinity College, University of Cambridge, United Kingdom}
\date{June 6, 2026}
\thanks{Email addresses: \href{mailto:el593@cam.ac.uk}{el593@cam.ac.uk}, \href{mailto:contact@ericli.com}{contact@ericli.com}.}


\hypersetup{
  pdftitle={Erdos Problem 684 at Density One: Small-prime Parts of Binomial Coefficients and Gaussian Fluctuations},
  pdfauthor={Eric Li},
  pdfsubject={Erdos Problem 684 at density one; small-prime parts of binomial coefficients},
  pdfkeywords={Erdos Problem 684, binomial coefficients, Kummer theorem, prime factors, normal order, central limit theorem, small-prime parts}
}

\begin{document}

\begin{abstract}
For $0\leq k\leq n$, let $u(n,k)$ be the largest divisor of $\binom nk$ whose prime factors are at most $k$.  Erd\H{o}s Problem \#684 concerns the special threshold $u(n,k)>n^2$ and asks how early this small-prime part can be forced to become large.  We prove the density-one analogue for every fixed power threshold.  If $f_c(n)$ is the least $k$ for which $u(n,k)>n^c$, then, for each fixed $c>0$,
\[
        f_c(n)=\left(\frac{c}{1-\gamma}+o(1)\right)\log n
\]
for almost all positive integers $n$.  In particular,
\[
        f_2(n)=\left(\frac{2}{1-\gamma}+o(1)\right)\log n
        =(4.730544237\ldots+o(1))\log n
\]
for the Erd\H{o}s \#684 threshold.  This is a normal-order theorem, not a pointwise resolution of the corresponding worst-case problem.

The constant $1-\gamma$ is arithmetic.  Kummer's theorem rewrites $\log u(n,k)$ as a sum of carry indicators, and complete-residue averaging gives
\[
        m(k)=k\sum_{p\leq k}\frac{\log p}{p-1}-\log k!=(1-\gamma)k+o(k).
\]
The cancellation in this formula moves the typical crossing from the naive scale $c\log n$ to $c(1-\gamma)^{-1}\log n$.  We prove the required concentration uniformly for every $k\leq A\log X$ on one dyadic interval, after discarding a zero-density exceptional set caused by large powers of small primes dividing one of the nearby integers $n,n-1,\ldots$.

We also prove Gaussian fluctuations in the logarithmic range.  If $k=k(X)\to\infty$, $k\leq A\log X$, and $n$ is uniform in $[X,2X)\cap\mathbb Z$, then
\[
        \frac{\log u(n,k)-m(k)}{\sqrt{V(k)}}\Rightarrow \mathcal N(0,1),
        \qquad
        V(k)\sim (2-\log(2\pi))k\log k.
\]
Higher prime powers are needed for the mean, but after centering their aggregate is $L^2$-negligible on the Gaussian scale; the variance comes only from the prime levels.
\end{abstract}

\maketitle

\section{Introduction}

All logarithms are natural.  For a prime $p$, write $\vp_p$ for the $p$-adic valuation.  For $0\leq k\leq n$, define
\[
        u(n,k):=\prod_{p\leq k}p^{\vp_p\binom nk}.
\]
Thus $u(n,k)$ is the largest divisor of $\binom nk$ all of whose prime factors are at most $k$.  In particular, $u(n,0)=u(n,1)=1$.  For a fixed real number $c>0$, put
\[
        f_c(n):=\min\{0\leq k\leq n:u(n,k)>n^c\},
\]
with the convention that $f_c(n)=\infty$ if the set is empty.

The case $c=2$ is a density-one counterpart of Erd\H{o}s Problem \#684.  In that problem one writes, for every $0\leq k\leq n$,
\[
        \binom nk=uv,
\]
where the primes dividing $u$ are in $[2,k]$ and the primes dividing $v$ are in $(k,n]$, and asks for bounds on the least $k$ for which $u>n^2$; see Erd\H{o}s \cite{Erdos1979} and the problem record \cite{Bloom684}.  The theorem below addresses the natural-density-one version of a more general threshold $u(n,k)>n^c$.  It determines the main term of the first crossing for almost all $n$, but it should not be read as a worst-case estimate.  Recent related work gives polylogarithmic worst-case upper bounds and logarithmic lower-bound examples; see Alexeev, Putterman, Sawhney, Sellke, and Valiant \cite{APSSV2026}.

The density-one formulation is not a cosmetic weakening of a pointwise assertion.  The proof removes the exceptional integers $n$ for which a large power of a small prime divides one of the short list of nearby integers $n-b$, $0\leq b\leq A\log X$.  These congruence events have total density tending to zero, but individual worst-case integers can be dominated by exactly this kind of local obstruction.  Thus the result should be viewed as a normal-order theorem for the first small-prime crossing, analogous in spirit to normal-order statements for additive arithmetic functions rather than to a uniform bound valid on every input.

\begin{theorem}\label{thm:main}
Let $c>0$ be fixed.  For $\eta>0$, let $E_{c,\eta}(N)$ be the set of integers $2\leq n\leq N$ such that either $f_c(n)=\infty$, or $f_c(n)<\infty$ and
\[
        \left|\frac{f_c(n)}{\log n}-\frac{c}{1-\gamma}\right|>\eta.
\]
Then
\[
        \lim_{N\to\infty}\frac{\#E_{c,\eta}(N)}{N}=0.
\]
Equivalently,
\[
        f_c(n)=\left(\frac{c}{1-\gamma}+o(1)\right)\log n
\]
for almost all positive integers $n$.
\end{theorem}

Taking $c=2$ gives the density-one consequence for the original Erd\H{o}s threshold.

\begin{corollary}\label{cor:erdos684}
For almost all positive integers $n$,
\[
        f_2(n)=\left(\frac{2}{1-\gamma}+o(1)\right)\log n
        =(4.730544237\ldots+o(1))\log n.
\]
\end{corollary}

Our second result records the Gaussian fluctuations of the carry sum at a fixed logarithmic value of $k$.  Put
\[
        U_k(n):=\log u(n,k).
\]
For a real number $X\geq3$, let
\[
        \II_X:=\{n\in\ZZ:X\leq n<2X\},
\]
and let $\PP_X$ and $\EE_X$ denote probability and expectation when $n$ is chosen uniformly from $\II_X$.  For an integer $q\geq1$, $[x]_q$ denotes the least non-negative residue of the integer $x$ modulo $q$.

For $k\geq1$, define the complete-residue mean
\[
        m(k):=\sum_{p\leq k}\log p\sum_{a\geq1}\frac{[k]_{p^a}}{p^a}.
\]
The inner series defining $m(k)$ is convergent, since $0\leq [k]_{p^a}/p^a\leq k/p^a$.  For $k\geq2$, define
\[
        \alpha_p(k):=\frac{[k]_p}{p}=\fracp{\frac{k}{p}},
        \qquad
        V(k):=\sum_{p\leq k}(\log p)^2\alpha_p(k)(1-\alpha_p(k)).
\]

\begin{theorem}[Gaussian fluctuations]\label{thm:clt}
Fix $A>0$.  Let $k=k(X)$ be an integer-valued function such that
\[
        k\to\infty,
        \qquad
        k\leq A\log X.
\]
Then, as $X\to\infty$,
\[
        V(k)=(2-\log(2\pi)+o(1))k\log k
\]
and
\[
        \frac{U_k(n)-m(k)}{\sqrt{V(k)}}\Rightarrow \normal(0,1)
        \qquad(n\text{ uniform in }\II_X).
\]
In addition,
\[
        \EE_X U_k(n)=m(k)+o(\sqrt{V(k)}),
        \qquad
        \Var_X(U_k(n))\sim V(k),
\]
and consequently
\[
        \frac{U_k(n)-\EE_X U_k(n)}{\sqrt{\Var_X(U_k(n))}}
        \Rightarrow \normal(0,1).
\]
\end{theorem}

\subsection*{Proof strategy}

Kummer's theorem converts $\vp_p\binom nk$ into a count of carries, equivalently into residue inequalities.  The exact form used here is
\begin{equation}\label{eq:kummer-intro}
        U_k(n)=\sum_{p\leq k}\sum_{a\geq1}(\log p)\one_{[n]_{p^a}<[k]_{p^a}}
\end{equation}
whenever $0\leq k\leq n$.  A complete-residue average of the summand with modulus $p^a$ is $[k]_{p^a}/p^a$, giving the deterministic mean $m(k)$.  A naive first-order estimate would suggest a coefficient $1$ in the crossing threshold, but the exact complete-residue mean satisfies
\[
        m(k)=k\sum_{p\leq k}\frac{\log p}{p-1}-\log k!=(1-\gamma)k+o(k).
\]
This cancellation is responsible for the constant $c/(1-\gamma)$ in Theorem \ref{thm:main}.  The key point is that the two large terms $k\log k$ and $\log k!$ cancel, leaving a linear main term.  Once $U_k(n)$ is known to be uniformly close to this mean for all $k$ in a logarithmic window, the first crossing of the level $c\log n$ must occur near $k=c(1-\gamma)^{-1}\log n$.

Two technical points are central.  First, the normal-order theorem requires concentration simultaneously for every integer $k\leq A\log X$ on the same dyadic interval.  We remove, outside a single exceptional set of size $o(X)$, all contributions from prime powers $p^a$ that are too large for a fourth-moment expansion.  After this removal, every least common multiple arising from four residue conditions is $o(X)$, so interval averages may be replaced by complete-residue averages with a summable error.  This is the role of the truncation $p^a\leq X^{1/10}$; the numerical exponent is immaterial, but it must be smaller than $1/4$ because four moduli appear in the moment expansion.

Secondly, the central limit theorem has two different scales.  Higher prime powers contribute a deterministic amount to $m(k)$, and so they cannot simply be discarded before centering.  After centering, however, their total $L^2$ size is $o(\sqrt{k\log k})$.  Thus the variance on the Gaussian scale comes only from the prime levels $p$, where the residue classes for distinct primes factor through the Chinese remainder theorem and lead to an independent Bernoulli model.

The argument never uses monotonicity of $U_k(n)$ in $k$; indeed both the binomial coefficient and the set of permitted primes change with $k$.  The proof of Theorem \ref{thm:main} instead combines a uniform lower exclusion below the proposed crossing point with a single exhibited value of $k$ above it.

Throughout, $F\ll_A G$ means $|F|\leq C_A G$ for a constant depending at most on $A$.  Unless stated otherwise, asymptotic notation in dyadic estimates is as $X\to\infty$, with fixed parameters such as $A$, $c$, and $\eta$ held fixed.  Supremums over ranges of $k$ are over integer values of $k$.

\section{Standard estimates, the carry formula, and the mean}\label{sec:mean}

We shall use the following standard analytic estimates.

\begin{lemma}[Standard estimates]\label{lem:standard_estimates}
Let
\[
        \theta(x):=\sum_{p\leq x}\log p.
\]
As $x\to\infty$,
\[
        \theta(x)=x+o(x)
\]
and
\[
        \sum_{m\leq x}\frac{\von(m)}m=\log x-\gamma+o(1).
\]
Moreover, for each fixed integer $j\geq1$ and all $x\geq2$,
\[
        \sum_{p\leq x}(\log p)^j\ll_j x(\log x)^{j-1},
        \qquad
        \pi(x)\ll \frac{x}{\log(2x)},
        \qquad
        \sum_{p\leq x}p\log p\ll x^2.
\]
\end{lemma}

\begin{proof}
The first estimate is the prime number theorem in Chebyshev's $\theta$-function form.  The second is the Mertens-von Mangoldt estimate with its constant term; see Montgomery and Vaughan \cite[Chs.~1--2]{MV}.  Chebyshev's estimate gives $\pi(x)\ll x/\log(2x)$.  The displayed bound for $\sum_{p\leq x}(\log p)^j$ follows from Chebyshev's estimate and partial summation.  Finally,
\[
        \sum_{p\leq x}p\log p
        \leq x\sum_{p\leq x}\log p
        =x\theta(x)\ll x^2.
\]
\end{proof}

We next translate divisibility of a binomial coefficient into residue inequalities.  This is the form of Kummer's theorem \cite{Kummer1852} needed in the sequel; the proof is included to fix the residue convention.

\begin{lemma}[Kummer's formula in residue form]\label{lem:kummer}
For every prime $p$ and every pair of integers $0\leq k\leq n$,
\[
        \vp_p\binom nk
        =\sum_{a\geq1}\one_{[k]_{p^a}>[n]_{p^a}}.
\]
Consequently,
\[
        U_k(n)=\sum_{p\leq k}\sum_{a\geq1}
        (\log p)\one_{[n]_{p^a}<[k]_{p^a}}.
\]
\end{lemma}

\begin{proof}
The sums are finite for fixed $n$ and $k$: if $p^a>n$, then $[n]_{p^a}=n\geq k=[k]_{p^a}$, so the corresponding indicator is zero.

Legendre's formula gives
\[
\vp_p\binom nk
=\sum_{a\geq1}\left(\left\lfloor\frac n{p^a}\right\rfloor
-\left\lfloor\frac k{p^a}\right\rfloor
-\left\lfloor\frac{n-k}{p^a}\right\rfloor\right).
\]
Fix $q=p^a$ and write $n=qN+r$ and $k=qK+s$, with $0\leq r,s<q$.  Then $n-k=q(N-K)+(r-s)$.  If $s\leq r$, then $0\leq r-s<q$, and the summand above is $N-K-(N-K)=0$.  If $s>r$, then $-q<r-s<0$, and the summand is $N-K-(N-K-1)=1$.  Therefore the summand is exactly $\one_{s>r}=\one_{[k]_q>[n]_q}$.  Multiplying by $\log p$ and summing over $p\leq k$ proves the displayed formula for $U_k(n)$.
\end{proof}

For a complete set of residues modulo $p^a$, the average of $\one_{[n]_{p^a}<[k]_{p^a}}$ is $[k]_{p^a}/p^a$.  The next lemma computes the main term of this average.

\begin{lemma}\label{lem:mean_arithmetic}
As $k\to\infty$,
\[
        m(k)=(1-\gamma)k+o(k).
\]
Moreover, for every fixed $A>0$,
\[
        \sup_{1\leq k\leq A L}|m(k)-(1-\gamma)k|=o_A(L)
        \qquad (L\to\infty).
\]
\end{lemma}

\begin{proof}
For every positive integer $q$,
\[
        [k]_q=k-q\left\lfloor\frac{k}{q}\right\rfloor,
        \qquad
        \frac{[k]_q}{q}=\frac{k}{q}-\left\lfloor\frac{k}{q}\right\rfloor.
\]
Using this with $q=p^a$ gives
\begin{align*}
        m(k)
        &=k\sum_{p\leq k}\log p\sum_{a\geq1}\frac1{p^a}
          -\sum_{p\leq k}\log p\sum_{a\geq1}\left\lfloor\frac{k}{p^a}\right\rfloor  \\
        &=k\sum_{p\leq k}\frac{\log p}{p-1}-\log k!.
\end{align*}
The last identity is Legendre's formula for the exponent of $p$ in $k!$, summed with weight $\log p$.

It remains to estimate the prime sum.  We have
\[
        \sum_{p\leq k}\frac{\log p}{p-1}
        =\sum_{p\leq k}\sum_{a\geq1}\frac{\log p}{p^a}.
\]
The part of this double sum with $p^a>k$ is $o(1)$.  Indeed, if $p\leq k^{1/2}$ and $b$ is the largest integer with $p^b\leq k$, then
\[
        \sum_{a>b}\frac{\log p}{p^a}\ll \frac{\log p}{k},
\]
because $p^{b+1}>k$.  Summing over $p\leq k^{1/2}$ gives $O(k^{-1}\theta(k^{1/2}))=O(k^{-1/2})$.  If $k^{1/2}<p\leq k$, the omitted terms have $a\geq2$ and contribute
\[
        \ll \sum_{k^{1/2}<p\leq k}\frac{\log p}{p^2}
        \leq \sum_{k^{1/2}<m\leq k}\frac{\log k}{m^2}
        \ll \frac{\log k}{k^{1/2}}.
\]
Therefore, by Lemma \ref{lem:standard_estimates},
\[
        \sum_{p\leq k}\frac{\log p}{p-1}
        =\sum_{p^a\leq k}\frac{\log p}{p^a}+o(1)
        =\sum_{m\leq k}\frac{\von(m)}m+o(1)
        =\log k-\gamma+o(1).
\]
Stirling's formula gives $\log k!=k\log k-k+o(k)$, and hence
\[
        m(k)=k(\log k-\gamma+o(1))-(k\log k-k+o(k))=(1-\gamma)k+o(k).
\]

For the uniform statement, fix $\delta>0$ and choose $K_0$ so that
\[
        |m(k)-(1-\gamma)k|\leq \frac{\delta}{A+1}k
        \qquad(k\geq K_0).
\]
Then the error is at most $\delta L$ for $K_0\leq k\leq A L$.  The finitely many values $1\leq k<K_0$ contribute $O_{K_0}(1)=o(L)$.  Since $\delta$ is arbitrary, the supremum is $o_A(L)$.
\end{proof}

\section{Removing large prime powers}\label{sec:tail-removal}

The concentration argument will use complete-residue averages for products of four residue functions.  To make the periods of those products small compared with the interval length, we first discard prime-power levels above a fixed power of $X$.

The discarded levels cannot be ignored for every integer $n$.  If $q=p^a$ is large while $k$ is logarithmic, then the event $[n]_q<[k]_q$ says that $n$ lies very close to a multiple of $q$.  Equivalently, $q$ divides one of $n,n-1,\ldots,n-k+1$.  The following lemma shows that, after taking a union over all small bases $p$ and all logarithmic shifts $b$, the set of $n$ for which this happens has density tending to zero.

Fix $A>0$.  For a dyadic parameter $X\geq3$, put
\[
        L:=\log X,
        \qquad
        Q_0:=X^{1/10}.
\]
For $1\leq k\leq A L$ define the truncated carry sum
\[
        U_k^{\leq Q_0}(n):=
        \sum_{\substack{p\leq k,\ a\geq1\\ p^a\leq Q_0}}
        (\log p)\one_{[n]_{p^a}<[k]_{p^a}}.
\]
The exponent $1/10$ is not essential.  Any fixed exponent $\vartheta<1/4$ would suffice for the fourth-moment argument, because a least common multiple of four truncated moduli would then be at most $X^{4\vartheta}=o(X)$.

\begin{lemma}\label{lem:tail}
For each fixed $A>0$, for all but $o_A(X)$ integers $n\in[X,2X)$, one has
\[
        U_k(n)=U_k^{\leq Q_0}(n)
\]
simultaneously for every integer $1\leq k\leq A\log X$.
\end{lemma}

\begin{proof}
Assume $X$ is large enough that $Q_0>A L$.  Let $\mathcal E_A(X)$ be the set of integers $n\in[X,2X)$ for which there exist a prime $p\leq A L$, an integer $a\geq1$, and an integer $b$ with $0\leq b\leq A L$ such that
\[
        Q_0<p^a\leq 2X,
        \qquad
        p^a\mid n-b.
\]
We first show that $\#\mathcal E_A(X)=o_A(X)$.

For fixed $q=p^a$ and $b$, the number of $n\in[X,2X)$ satisfying $q\mid n-b$ is $O(X/q+1)$.  There are $O_A(L)$ possible integers $b$.  Therefore
\[
\#\mathcal E_A(X)
\ll_A L\sum_{p\leq A L}\sum_{\substack{a\geq1\\ Q_0<p^a\leq 2X}}
        \left(\frac X{p^a}+1\right).
\]
After division by $X$, this is at most
\[
L\sum_{p\leq A L}\sum_{p^a>Q_0}\frac1{p^a}
+\frac{L}{X}\sum_{p\leq A L}\#\{a\geq1:p^a\leq2X\}.
\]
For each fixed prime $p\leq A L<Q_0$, the geometric tail satisfies
\[
        \sum_{p^a>Q_0}p^{-a}\ll Q_0^{-1}
\]
uniformly in $p$.  Hence the first term is $O_A(L^2/Q_0)=o_A(1)$.  Also
\[
        \#\{a\geq1:p^a\leq2X\}\leq \frac{\log(2X)}{\log 2}=O(L),
\]
so the second term is $O_A(L^3/X)=o_A(1)$.  Thus $\#\mathcal E_A(X)=o_A(X)$.

Now suppose that $n\notin\mathcal E_A(X)$ and that a level $q=p^a>Q_0$ contributes to $U_k(n)$ for some integer $k\leq A L$.  Since $q>Q_0>A L\geq k$, we have $[k]_q=k$.  The contribution condition is then $[n]_q<k\leq A L$.  Put $b=[n]_q$.  If $q>2X$, then $[n]_q=n\geq X>k$, so no contribution is possible.  Hence any contributing $q$ satisfies $q\leq2X$, and $q\mid n-b$ with $0\leq b\leq A L$.  This places $n$ in $\mathcal E_A(X)$, a contradiction.  Therefore no level $q>Q_0$ contributes for any $k\leq A L$ outside $\mathcal E_A(X)$.
\end{proof}

\section{Periodic averages and one-prime tower moments}\label{sec:periodic-tower}

We shall repeatedly replace interval averages by complete-residue averages.  The following elementary estimate is the only input needed for this replacement.

\begin{lemma}[Periodic averaging]\label{lem:periodic_average}
If $F$ is periodic modulo $Q$ and $|F|\leq1$, then
\[
        \EE_X F(n)
        =\frac1Q\sum_{n\bmod Q}F(n)+O\left(\frac QX+\frac1X\right),
\]
uniformly for $Q\geq1$ and $X\geq1$.  In particular, when $Q\leq X$ the error is $O(Q/X)$.
\end{lemma}

\begin{proof}
Let $M_X:=\#\II_X$.  Then $M_X=X+O(1)$.  Decompose $\II_X$ into complete consecutive blocks of $Q$ integers, together with at most two incomplete end blocks.  On each complete block the average of $F$ is exactly $Q^{-1}\sum_{n\bmod Q}F(n)$.  The incomplete end blocks contain $O(Q+1)$ integers altogether.  Since $|F|\leq1$, their contribution to the normalized average is $O((Q+1)/M_X)$, which is $O(Q/X+1/X)$.
\end{proof}

The next elementary refinement is used only for the very sparse tail in the proof of Lemma \ref{lem:higher_powers}.  It avoids losing a factor $q$ from an incomplete block when the permitted residue classes form a short initial interval.

\begin{lemma}[Short residue windows]\label{lem:short_window}
Let $q\geq1$ and $1\leq h\leq q$ be integers.  Then
\[
        \PP_X([n]_q<h)\ll \frac{h}{q}+\frac{h}{X},
\]
uniformly in $q$, $h$, and $X\geq1$.
\end{lemma}

\begin{proof}
For each residue $0\leq r<h$, the number of integers $n\in[X,2X)$ with $n\equiv r\pmod q$ is $X/q+O(1)$.  Summing this estimate over $0\leq r<h$ gives $O(hX/q+h)$ admissible integers.  Division by $\#\II_X=X+O(1)$ gives the claimed $O(h/q+h/X)$ bound, after changing the absolute constant if necessary.  The gain over the crude periodic estimate $O(h/q+q/X)$ is that the incomplete end blocks contain only the $h$ relevant residue classes, not all $q$ possible classes.
\end{proof}

\begin{lemma}[Chinese-remainder factorization]\label{lem:crt_factorization}
Let $\mathcal P$ be a finite set of primes.  For each $p\in\mathcal P$, let $b_p\geq1$ and let $F_p$ be a function on residues modulo $p^{b_p}$.  Put $Q:=\prod_{p\in\mathcal P}p^{b_p}$.  Then
\[
        \frac1Q\sum_{n\bmod Q}\prod_{p\in\mathcal P}F_p(n\bmod p^{b_p})
        =\prod_{p\in\mathcal P}\left(
        \frac1{p^{b_p}}\sum_{x\bmod p^{b_p}}F_p(x)\right).
\]
The same conclusion holds when $F_p$ depends only on residues modulo lower powers of $p$.
\end{lemma}

\begin{proof}
The Chinese remainder theorem identifies residues modulo $Q$ bijectively with tuples of residues modulo $p^{b_p}$ for $p\in\mathcal P$.  Under this identification, the average over $n\bmod Q$ becomes the product average over the prime-power components.
\end{proof}

For a modulus $q$ and an integer $k$, put
\[
        r_q(k):=[k]_q,
        \qquad
        W_q(n;k):=\one_{[n]_q<r_q(k)}-\frac{r_q(k)}q.
\]
Then $W_q(\cdot;k)$ is periodic modulo $q$, is bounded in absolute value by $1$, and has mean zero over a complete set of residues modulo $q$.

The next lemma controls the cumulative contribution of all powers of one fixed prime.  It is important that the contribution from levels $p^a>k$ has bounded moments: these events are nested, not independent.

\begin{lemma}[Moments in one prime tower]\label{lem:tower_moments}
Let $p\leq k$ be prime, let $b\geq1$, and let $N_p$ be uniform modulo $p^b$.  Define
\[
        Y_p:=\sum_{1\leq a\leq b}\left(\one_{N_p\bmod p^a<[k]_{p^a}}-\frac{[k]_{p^a}}{p^a}\right),
\]
and put
\[
        \alpha_p^*:=\left\lfloor\log_p k\right\rfloor.
\]
For $j=2,4$,
\[
        \EE |Y_p|^j\ll_j (\alpha_p^*+1)^j,
\]
with absolute implied constants.
\end{lemma}

\begin{proof}
There are at most $\alpha_p^*$ indices $a$ with $p^a\leq k$.  The sum of the corresponding centered indicators is therefore bounded in absolute value by $\alpha_p^*$.

For $a>\alpha_p^*$, one has $p^a>k$, and so $[k]_{p^a}=k$.  Let
\[
        E_a:=\{N_p\bmod p^a<k\},
        \qquad
        T_p:=\sum_{\alpha_p^*<a\leq b}\one_{E_a}.
\]
The events $E_a$ are nested decreasing for $a\geq\alpha_p^*+1$.  Indeed, if $a\geq\alpha_p^*+2$ and $E_a$ occurs, then
\[
        N_p\bmod p^a<k<p^{a-1},
\]
so reducing modulo $p^{a-1}$ leaves the same residue and gives $E_{a-1}$.  Hence
\[
        E_{\alpha_p^*+1}\supseteq E_{\alpha_p^*+2}\supseteq\cdots.
\]
It follows that, for $r\geq1$, the event $T_p\geq r$ implies $E_{\alpha_p^*+r}$ whenever $\alpha_p^*+r\leq b$, and is empty otherwise.  Therefore
\[
        \PP(T_p\geq r)\leq \frac{k}{p^{\alpha_p^*+r}}\leq p^{1-r}.
\]
The tail formula for moments of a non-negative integer-valued random variable gives, for $j=2,4$,
\[
        \EE T_p^j\ll_j \sum_{r\geq1}r^{j-1}\PP(T_p\geq r)
        \ll_j \sum_{r\geq1}r^{j-1}p^{1-r}
        \ll_j 1.
\]
The expectation of the high-level uncentered sum is also bounded:
\[
        \sum_{a>\alpha_p^*}\frac{k}{p^a}
        \leq \frac{k}{p^{\alpha_p^*+1}}\frac1{1-1/p}\ll 1.
\]
Therefore, by $|T_p-\EE T_p|^j\ll_j T_p^j+(\EE T_p)^j$, the centered high-level contribution has bounded second and fourth moments.  Combining this with the deterministic bound for the levels $p^a\leq k$ proves the lemma.
\end{proof}

\section{A fourth moment and uniform concentration}\label{sec:fourth}

For $Q_0=X^{1/10}$, define the truncated complete-residue mean
\[
        m_{Q_0}(k):=
        \sum_{\substack{p\leq k,\ a\geq1\\ p^a\leq Q_0}}
        (\log p)\frac{[k]_{p^a}}{p^a}.
\]

\begin{lemma}[Fourth moment for the truncated sum]\label{lem:fourth}
For each fixed $A>0$, uniformly for integers $1\leq k\leq A\log X$,
\[
\EE_X\left|
U_k^{\leq Q_0}(n)-m_{Q_0}(k)
\right|^4
        \ll_A k^2(\log(2k))^2+1.
\]
\end{lemma}

\begin{proof}
Write
\[
        S_k(n):=U_k^{\leq Q_0}(n)-m_{Q_0}(k)
        =\sum_{\substack{p\leq k,\ a\geq1\\p^a\leq Q_0}}
        (\log p)W_{p^a}(n;k).
\]
Expanding $S_k(n)^4$ gives a sum over quadruples $q_i=p_i^{a_i}\leq Q_0$, with $p_i\leq k$.  We first replace the average over the interval $[X,2X)$ by the average over a complete system of residues modulo the least common multiple of the four moduli.  This reduction is harmless precisely because the truncation makes every such least common multiple much smaller than $X$.  For such a quadruple set $Q(\mathbf q)=\lcm(q_1,q_2,q_3,q_4)$.  Then $Q(\mathbf q)\leq Q_0^4=X^{2/5}$.  By Lemma \ref{lem:periodic_average}, applied to
\[
        F_{\mathbf q}(n):=\prod_{i=1}^4 W_{q_i}(n;k),
\]
the total error made by replacing all interval averages by complete-residue averages is at most
\begin{align*}
&\frac1X
        \sum_{q_1,q_2,q_3,q_4}
        (\log p_1)(\log p_2)(\log p_3)(\log p_4)
        \lcm(q_1,q_2,q_3,q_4) \\
&\qquad +
\frac1X\left(
        \sum_{\substack{p\leq k,\ a\geq1\\p^a\leq Q_0}}
        \log p\right)^4,
\end{align*}
where the sums over $q_i$ are restricted by $q_i=p_i^{a_i}\leq Q_0$ and $p_i\leq k$.  Since the least common multiple is at most $q_1q_2q_3q_4$, the first term is bounded by
\[
        \frac1X\left(
        \sum_{\substack{p\leq k,\ a\geq1\\p^a\leq Q_0}}
        (\log p)p^a\right)^4.
\]
For each fixed $p$, $\sum_{p^a\leq Q_0}p^a\ll Q_0$, and $\theta(k)\ll k$.  Hence
\[
        \sum_{\substack{p\leq k,\ a\geq1\\p^a\leq Q_0}}(\log p)p^a
        \ll Q_0\theta(k)\ll Q_0k.
\]
Therefore the first error is $O_A(Q_0^4k^4/X)=o_A(1)$ uniformly for $k\leq A\log X$.  The $O(1/X)$ part of the periodic-averaging error is even smaller; for example,
\[
        \frac1X\left(
        \sum_{\substack{p\leq k,\ a\geq1\\p^a\leq Q_0}}
        \log p\right)^4
        \ll_A \frac{(k\log Q_0)^4}{X}=o_A(1),
\]
because each prime contributes at most $O(\log Q_0)$ levels after weighting by $\log p$.

It remains to bound the complete-residue fourth moment.  For each prime $p\leq k$, let $b_p$ be the largest integer with $p^{b_p}\leq Q_0$.  For all sufficiently large $X$, $b_p\geq1$ because $p\leq k\leq A\log X<Q_0$.  Let $N_p$ be uniform modulo $p^{b_p}$, with the variables $N_p$ independent for distinct primes $p$, and define
\[
        Z_p:=\log p\sum_{1\leq a\leq b_p}
        \left(
        \one_{N_p\bmod p^a<[k]_{p^a}}-
        \frac{[k]_{p^a}}{p^a}
        \right).
\]
The $Z_p$ are independent and have mean zero.  Lemma \ref{lem:crt_factorization} shows that the complete-residue contribution to the fourth moment is
\[
        \EE\left|\sum_{p\leq k}Z_p\right|^4.
\]
Repeated appearances of the same prime in a quadruple remain coupled through the same prime-power residue, while distinct primes factor independently; this is precisely what the variables $Z_p$ model.

By Lemma \ref{lem:tower_moments},
\[
        \EE |Z_p|^2\ll (\log p)^2(\alpha_p^*+1)^2,
        \qquad
        \EE |Z_p|^4\ll (\log p)^4(\alpha_p^*+1)^4,
\]
where $\alpha_p^*=\lfloor\log_p k\rfloor$.  Since $(\alpha_p^*+1)\log p\ll \log(2k)$ and $\pi(k)\ll k/\log(2k)$ for $k\geq2$,
\[
        \sum_{p\leq k}\EE |Z_p|^2\ll k\log(2k)+1,
        \qquad
        \sum_{p\leq k}\EE |Z_p|^4\ll k(\log(2k))^3+1.
\]
For independent mean-zero real random variables,
\[
        \EE\left|\sum_{p\leq k}Z_p\right|^4
        =\sum_{p\leq k}\EE Z_p^4
          +6\sum_{p<q}\EE Z_p^2\,\EE Z_q^2.
\]
Consequently,
\[
        \EE\left|\sum_{p\leq k}Z_p\right|^4
        \ll
        \left(\sum_{p\leq k}\EE |Z_p|^2\right)^2
        +\sum_{p\leq k}\EE |Z_p|^4
        \ll k^2(\log(2k))^2+1.
\]
Together with the interval-to-complete-residue error, this proves the lemma.
\end{proof}

\begin{lemma}[The truncated mean is the full mean]\label{lem:mean_interval}
Uniformly for $1\leq k\leq A\log X$,
\[
        m_{Q_0}(k)=m(k)+o_A(1).
\]
Moreover,
\[
        \EE_X U_k^{\leq Q_0}(n)=m(k)+o_A(1)
\]
uniformly for $1\leq k\leq A\log X$.
\end{lemma}

\begin{proof}
The part of $m(k)$ omitted from $m_{Q_0}(k)$ is at most
\[
\sum_{p\leq k}\log p\sum_{p^a>Q_0}\frac{[k]_{p^a}}{p^a}
\leq k\sum_{p\leq k}\log p\sum_{p^a>Q_0}\frac1{p^a}
        \ll_A \frac{k^2}{Q_0}=o_A(1),
\]
uniformly for $k\leq A\log X$.

For each $q=p^a\leq Q_0$, Lemma \ref{lem:periodic_average} gives
\[
\EE_X\one_{[n]_q<[k]_q}
        =\frac{[k]_q}{q}+O\left(\frac qX+\frac1X\right).
\]
Therefore the total interval-averaging error is
\[
        \ll \frac1X\sum_{\substack{p\leq k,\ a\geq1\\p^a\leq Q_0}}(\log p)p^a
        +\frac1X\sum_{\substack{p\leq k,\ a\geq1\\p^a\leq Q_0}}\log p
        \ll_A \frac{Q_0 k}{X}+\frac{k\log Q_0}{X}=o_A(1).
\]
The two assertions follow.
\end{proof}

\begin{proposition}[Uniform concentration]\label{prop:uniform}
Fix $A>0$.  For every fixed $\delta>0$,
\[
\#\left\{X\leq n<2X:
\sup_{1\leq k\leq A\log X}|U_k(n)-(1-\gamma)k|>\delta\log X
\right\}=o_{A,\delta}(X),
\]
where the supremum is over integer $k$.  More quantitatively, with $L=\log X$,
\[
\begin{aligned}
&\#\left\{X\leq n<2X:
        \sup_{1\leq k\leq A L}|U_k(n)-m(k)|>\delta L
\right\} \\
&\qquad\qquad\ll_{A,\delta}X\frac{(\log L)^2}{L}+R_A(X),
\end{aligned}
\]
where $R_A(X)=o_A(X)$ is the exceptional-set size from Lemma \ref{lem:tail}.
\end{proposition}

\begin{proof}
By Lemma \ref{lem:tail}, outside a set of size $R_A(X)=o_A(X)$ we may replace every $U_k(n)$ by $U_k^{\leq Q_0}(n)$, simultaneously for all integers $1\leq k\leq A L$.

For fixed $k\leq A L$, Lemmas \ref{lem:fourth} and \ref{lem:mean_interval} imply, for all sufficiently large $X$,
\[
\EE_X|U_k^{\leq Q_0}(n)-m(k)|^4
        \ll_A L^2(\log L)^2.
\]
Markov's inequality gives
\[
\#\{X\leq n<2X:|U_k^{\leq Q_0}(n)-m(k)|>\delta L\}
        \ll_{A,\delta}X\frac{(\log L)^2}{L^2}.
\]
There are $O_A(L)$ integer values of $k$ in the range $1\leq k\leq A L$.  Taking the union over these values gives
\[
        \ll_{A,\delta}X\frac{(\log L)^2}{L}
\]
exceptions, in addition to the set from Lemma \ref{lem:tail}.  This is the only point at which a fourth moment, rather than only a second moment, is needed: after Markov's inequality the probability for one fixed $k$ must be summable over $O(L)$ possible values of $k$.  This proves the quantitative estimate around $m(k)$.

Finally, Lemma \ref{lem:mean_arithmetic} gives
\[
        \sup_{1\leq k\leq A L}|m(k)-(1-\gamma)k|=o_A(L).
\]
For sufficiently large $X$, this deterministic error is at most $\delta L/2$.  Applying the estimate around $m(k)$ with threshold $\delta L/2$ gives the asserted concentration around $(1-\gamma)k$.
\end{proof}

\section{Proof of the normal-order theorem}\label{sec:normal-order}

\begin{proof}[Proof of Theorem \ref{thm:main}]
Fix $c>0$ and put
\[
        C:=\frac{c}{1-\gamma}.
\]
Let $\eta>0$.  Choose a number $\sigma$ with
\[
        0<\sigma<\min\{C/2,\eta/4\},
\]
and set $A=C+2\sigma$.  Choose also
\[
        0<\delta<\frac{(1-\gamma)\sigma}{4}.
\]
By Proposition \ref{prop:uniform}, for all but $o(X)$ integers $n\in[X,2X)$,
\begin{equation}\label{eq:main-proof-uniform}
        |U_k(n)-(1-\gamma)k|\leq \delta\log X
\end{equation}
for every integer $1\leq k\leq A\log X$.

Consider an $n$ satisfying \eqref{eq:main-proof-uniform}.  If $1\leq k\leq (C-\sigma)\log X$, then
\[
        U_k(n)\leq (1-\gamma)(C-\sigma)\log X+
        \delta\log X
        \leq c\log X-\frac34(1-\gamma)\sigma\log X.
\]
Since $\log n\geq\log X$ for $n\in[X,2X)$, the last expression is smaller than $c\log n$ for all sufficiently large $X$.  The values $k=0$ and $k=1$ also do not satisfy $u(n,k)>n^c$, since $u(n,k)=1$ for them.  Hence no integer $k\leq (C-\sigma)\log X$ can satisfy the threshold inequality, and
\[
        f_c(n)>(C-\sigma)\log X.
\]

For the reverse inequality, set
\[
        k_+:=\left\lfloor (C+\sigma)\log X\right\rfloor.
\]
For all sufficiently large $X$, $1\leq k_+\leq A\log X$ and $k_+\leq X\leq n$.  The same uniform estimate gives
\begin{align*}
        U_{k_+}(n)
        &\geq (1-\gamma)k_+ -\delta\log X  \\
        &\geq c\log X +(1-\gamma)\sigma\log X -O(1)-\delta\log X  \\
        &\geq c\log X +\frac34(1-\gamma)\sigma\log X-O(1).
\end{align*}
Because $\log n\leq\log X+\log 2$, the positive term $\frac34(1-\gamma)\sigma\log X$ eventually dominates the bounded difference $c(\log n-\log X)$ and the $O(1)$ term.  Thus $U_{k_+}(n)>c\log n$ for all sufficiently large $X$.  Consequently $u(n,k_+)>n^c$, so $f_c(n)$ is finite and
\[
        f_c(n)\leq k_+\leq (C+\sigma)\log X.
\]

Combining the two bounds and using $\log X\leq\log n\leq\log X+
\log 2$, we obtain, for all sufficiently large $X$ and all non-exceptional $n\in[X,2X)$,
\[
        \frac{(C-\sigma)\log X}{\log X+\log 2}
        <\frac{f_c(n)}{\log n}
        \leq
        C+\sigma.
\]
Since $\sigma<\eta/4$, these inequalities imply
\[
        \left|\frac{f_c(n)}{\log n}-C\right|\leq \eta
\]
for all sufficiently large $X$.

It remains to pass from dyadic intervals to natural density.  Let $E_j$ be the exceptional set in $[2^j,2^{j+1})$ for the fixed tolerance $\eta$.  The dyadic estimate gives $|E_j|=o(2^j)$.  If $N\in[2^J,2^{J+1})$, then
\[
        \sum_{j<J}|E_j|=o(2^J)=o(N).
\]
Indeed, given $\varepsilon>0$, choose $J_0$ such that $|E_j|\leq\varepsilon 2^j$ for all $j\geq J_0$.  The contribution of $j<J_0$ is fixed, while the tail is at most $\varepsilon\sum_{J_0\leq j<J}2^j\leq\varepsilon 2^J$.  Since $\varepsilon$ is arbitrary, the full sum is $o(2^J)$.  The possible contribution from the final partial interval $[2^J,N]$ is bounded by the exceptional set in the full dyadic interval $[2^J,2^{J+1})$, and is again $o(N)$.  This proves the theorem.
\end{proof}

\begin{remark}
No monotonicity of $U_k(n)$ in $k$ is used.  This is essential because both the binomial coefficient $\binom nk$ and the set of primes $p\leq k$ vary with $k$.  The lower bound on $f_c(n)$ rules out every $k$ below $(C-\sigma)\log X$ by uniform concentration.  The upper bound then exhibits one specific value $k_+$ above $(C+\sigma)\log X$ for which the threshold inequality holds.
\end{remark}

\section{Gaussian fluctuations}\label{sec:clt}

This section proves Theorem \ref{thm:clt}.  We separate the prime levels $p$ from the higher prime-power levels $p^a$, $a\geq2$.  Only the prime levels survive on the scale $\sqrt{k\log k}$.

The separation is slightly delicate.  If the higher levels were omitted before centering, the mean would be wrong by an amount of order $k$ in general.  The statement therefore centers by the full complete-residue mean $m(k)$, and only then proves that the centered higher-level contribution is negligible.  The prime levels, in contrast, already have variance of order $k\log k$ and satisfy a triangular-array central limit theorem.

For $k\geq2$, define the prime-level centered sum
\[
        W_k(n):=\sum_{p\leq k}(\log p)
        \left(\one_{[n]_p<[k]_p}-\alpha_p(k)\right).
\]
The corresponding complete-residue variance is
\[
        V(k)=\sum_{p\leq k}(\log p)^2\alpha_p(k)(1-\alpha_p(k)).
\]

\begin{lemma}[The variance constant]\label{lem:variance_constant}
As $k\to\infty$,
\[
        V(k)=(\kappa+o(1))k\log k,
        \qquad
        \kappa:=2-\log(2\pi)>0.
\]
\end{lemma}

\begin{proof}
Let
\[
        \psi(y):=\fracp{y}\left(1-\fracp{y}\right).
\]
Since $[k]_p/p=\{k/p\}$, we have
\[
        V(k)=\sum_{p\leq k}(\log p)^2\psi(k/p).
\]
The function $\psi$ is bounded by $1/4$.  Hence, for a fixed integer $M\geq1$, the primes $p\leq k/(M+1)$ contribute at most
\[
        \ll \sum_{p\leq k/(M+1)}(\log p)^2
        \ll \frac{k\log k}{M+1},
\]
by Lemma \ref{lem:standard_estimates}.

It remains to treat $k/(M+1)<p\leq k$.  For this prime range,
\[
\sum_{k/(M+1)<p\leq k}(\log p)^2\psi(k/p)
=\int_{k/(M+1)}^k (\log t)\psi(k/t)\,d\theta(t).
\]
For fixed $M$, put
\[
        g_k(t):=(\log t)\psi(k/t)
        \qquad(k/(M+1)\leq t\leq k).
\]
We claim that $g_k$ has total variation $O_M(\log k)$.  The function $\psi(y)=\{y\}(1-\{y\})$ is continuous at integers, since both one-sided limits there are $0$.  Thus $g_k$ has no jump discontinuities at the points $t=k/j$.  On each interval
\[
        \frac{k}{j+1}<t<\frac{k}{j},
        \qquad 1\leq j\leq M,
\]
one has
\[
        \psi(k/t)=\left(\frac{k}{t}-j\right)\left(1-\frac{k}{t}+j\right),
\]
and hence $g_k'(t)=O_M((\log k)/t)$ on that interval.  Since there are only $M$ such intervals, the bounded-variation estimate follows.

By the prime number theorem,
\[
        \theta(t)=t+o(k)
\]
uniformly on $[k/(M+1),k]$ for fixed $M$.  Write $E(t)=\theta(t)-t$.  Stieltjes integration by parts for bounded-variation functions gives
\[
        \int_{k/(M+1)}^k g_k(t)\,dE(t)
        =O\left(\sup_{k/(M+1)\leq t\leq k}|E(t)|
        \bigl(\|g_k\|_\infty+\operatorname{Var}(g_k)\bigr)\right)
        =o_M(k\log k).
\]
Therefore
\[
\int_{k/(M+1)}^k g_k(t)\,d\theta(t)
=\int_{k/(M+1)}^k g_k(t)\,dt+o_M(k\log k).
\]
Changing variables $y=k/t$ gives
\[
\int_{k/(M+1)}^k (\log t)\psi(k/t)\,dt
        =k\int_1^{M+1}\frac{\log k-\log y}{y^2}\psi(y)\,dy.
\]
After division by $k\log k$ and passage to the limit $k\to\infty$ with $M$ fixed, this becomes
\[
        \int_1^{M+1}\frac{\psi(y)}{y^2}\,dy.
\]
The initial range $p\leq k/(M+1)$ contributes $O(1/M)$ after normalization; since $0\leq\psi\leq1/4$, the tail of the integral from $M+1$ to infinity is also $O(1/M)$.  Letting $M\to\infty$ gives
\[
        \lim_{k\to\infty}\frac{V(k)}{k\log k}
        =\int_1^\infty \frac{\fracp{y}(1-\fracp{y})}{y^2}\,dy.
\]

It remains only to evaluate this elementary integral.  On the interval $y=m+x$, with $m\geq1$ and $0\leq x<1$, one has $\{y\}=x$, and a direct integration gives
\[
\int_0^1\frac{x(1-x)}{(m+x)^2}\,dx
        =(2m+1)\log\left(1+\frac1m\right)-2.
\]
Therefore the partial integral over $1\leq y<N+1$ equals
\begin{align*}
\sum_{m=1}^N\left((2m+1)\log\left(1+\frac1m\right)-2\right)
&=(2N+1)\log(N+1)-2\log(N!)-2N.
\end{align*}
By Stirling's formula this tends to $2-\log(2\pi)$.  This proves the asserted asymptotic for $V(k)$ and the positivity of $\kappa$.
\end{proof}

\begin{lemma}[Prime-level central limit theorem]\label{lem:prime_clt}
Fix $A>0$.  If $k=k(X)\to\infty$ and $k\leq A\log X$, then
\[
        \frac{W_k(n)}{\sqrt{V(k)}}\Rightarrow \normal(0,1)
        \qquad(n\text{ uniform in }\II_X).
\]
Moreover, for every fixed integer $r\geq1$,
\[
        \EE_X\left(\frac{W_k(n)}{\sqrt{V(k)}}\right)^r
        \longrightarrow
        \begin{cases}
        0, & r\text{ odd},\\
        (r-1)!!, & r\text{ even}.
        \end{cases}
\]
\end{lemma}

\begin{proof}
Let $(B_p)_{p\leq k}$ be independent Bernoulli random variables with
\[
        \PP(B_p=1)=\alpha_p(k),
\]
and put
\[
        W_k^*:=\sum_{p\leq k}(\log p)(B_p-
        \alpha_p(k)).
\]
Then $\Var(W_k^*)=V(k)$.  Since
\[
        \max_{p\leq k}\frac{\log p}{\sqrt{V(k)}}
        \ll \sqrt{\frac{\log k}{k}}\to0
\]
by Lemma \ref{lem:variance_constant}, the Lindeberg condition for this triangular array is immediate.  Hence
\[
        \frac{W_k^*}{\sqrt{V(k)}}\Rightarrow \normal(0,1)
\]
by the Lindeberg-Feller central limit theorem; see, for example, Billingsley \cite[Th.~27.2]{BillingsleyPM}.

For every fixed $r$, the normalized $r$-th moments also converge to the Gaussian moments.  To see this, choose a fixed $s>\max\{r,2\}$.  Rosenthal's inequality \cite{Rosenthal1970} gives
\[
        \EE|W_k^*|^s
        \ll_s V(k)^{s/2}+\sum_{p\leq k}(\log p)^s.
\]
By Lemma \ref{lem:standard_estimates},
\[
        \sum_{p\leq k}(\log p)^s\ll_s k(\log k)^{s-1}=o(V(k)^{s/2})
\]
for fixed $s>2$.  Thus the normalized variables $W_k^*/\sqrt{V(k)}$ are uniformly integrable in order $r$, and convergence of $r$-th moments follows from convergence in distribution.

It remains to transfer the fixed moments from the independent model to the interval $\II_X$.  Fix an integer $r\geq1$ and expand $W_k(n)^r$.  For a tuple $(p_1,\ldots,p_r)$, the product of the corresponding centered residue functions is periodic modulo a divisor of $p_1\cdots p_r$ and is bounded in absolute value by $1$.  Lemma \ref{lem:periodic_average} shows that the interval average of this product differs from its complete-residue average by
\[
        O_r\left(\frac{p_1\cdots p_r}{X}+\frac1X\right).
\]
Multiplying by the weights and summing over all tuples gives the total error
\[
        \ll_r \frac1X\left(\sum_{p\leq k}p\log p\right)^r
        +\frac1X\left(\sum_{p\leq k}\log p\right)^r
        \ll_{r,A} \frac{(\log X)^{2r}}{X}
        =o(V(k)^{r/2}).
\]
The complete-residue average is exactly the corresponding moment of $W_k^*$, by Lemma \ref{lem:crt_factorization}: repeated appearances of the same prime use the same Bernoulli variable $B_p$, while distinct primes factor independently.  Therefore
\[
        \EE_X W_k(n)^r=\EE (W_k^*)^r+o(V(k)^{r/2})
\]
for every fixed $r$.  The displayed moment convergence follows.  Since the standard normal distribution is moment-determinate, the method of moments gives convergence in distribution.
\end{proof}

We now prove that higher prime powers are negligible at the same scale.  For $X$ sufficiently large in the range $k\leq A\log X$, define
\[
        R_k(n):=\sum_{p\leq k}\sum_{a\geq2}(\log p)
        \left(\one_{[n]_{p^a}<[k]_{p^a}}-
        \frac{[k]_{p^a}}{p^a}\right).
\]
For $n\geq k$ this series is absolutely convergent: for all sufficiently large $a$ the indicator is zero, and the deterministic tail is dominated by $\sum_a k/p^a$.

\begin{lemma}[Higher prime powers are negligible]\label{lem:higher_powers}
Fix $A>0$.  If $k=k(X)\to\infty$ and $k\leq A\log X$, then
\[
        \EE_X |R_k(n)|^2=o(k\log k).
\]
\end{lemma}

\begin{proof}
Put $L=\log X$ and
\[
        Q_1:=\exp(\sqrt L).
\]
Since $k\leq A L$, one has $Q_1>k$ for all sufficiently large $X$.  Split
\[
        R_k=R_k^{\leq Q_1}+R_k^{>Q_1}
\]
according to whether $p^a\leq Q_1$ or $p^a>Q_1$.

We first prove that the tail is negligible.  Let
\[
        T(n):=\sum_{\substack{p\leq k,\ a\geq2\\p^a>Q_1}}
        (\log p)\one_{[n]_{p^a}<[k]_{p^a}},
\]
and let
\[
        M:=\sum_{\substack{p\leq k,\ a\geq2\\p^a>Q_1}}
        (\log p)\frac{[k]_{p^a}}{p^a}.
\]
Then $R_k^{>Q_1}=T-M$.  Because $p^a>Q_1>k$, we have $[k]_{p^a}=k$ throughout this tail, and the geometric-tail bound gives
\[
        M\leq k\sum_{p\leq k}\log p\sum_{p^a>Q_1}\frac1{p^a}
        \ll \frac{k^2}{Q_1}=o(1).
\]
For a prime power $q=p^a>Q_1$, the event $[n]_q<k$ occupies the $k$ residue classes $0,1,\ldots,k-1$ modulo $q$.  It is important here not to use the cruder periodic averaging error $O(q/X)$, since in the tail $q$ may be much larger than $k$.  Lemma \ref{lem:short_window} gives the sharper estimate
\[
        \PP_X([n]_q<k)\ll \frac{k}{q}+\frac{k}{X}.
\]
If $q>2X$, then $[n]_q=n\geq X>k$, so the uncentered indicator is actually zero.  Hence
\begin{align*}
\EE_X T
&\ll k\sum_{p\leq k}\log p\sum_{p^a>Q_1}\frac1{p^a}
 +\frac{k}{X}\sum_{p\leq k}\sum_{\substack{a\geq2\\p^a\leq2X}}\log p  \\
&\ll \frac{k^2}{Q_1}+\frac{k^2L}{X\log k}=o(1).
\end{align*}
The second term uses
\[
        \sum_{p\leq k}\sum_{\substack{a\geq2\\p^a\leq2X}}\log p
        \ll \sum_{p\leq k}\frac{L}{\log p}\log p
        \ll L\pi(k)\ll \frac{kL}{\log k}.
\]
Moreover
\[
        0\leq T(n)\leq
        \sum_{p\leq k}\sum_{\substack{a\geq2\\p^a\leq2X}}\log p
        \ll \frac{kL}{\log k}.
\]
Combining this supremum bound with the displayed estimate for $\EE_XT$ gives
\[
        \EE_X T^2\leq \|T\|_\infty\EE_XT
        \ll \frac{kL}{\log k}\left(\frac{k^2}{Q_1}+\frac{k^2L}{X\log k}\right)=o(1),
\]
since $k\leq A L$ and $Q_1=\exp(\sqrt L)$.  Therefore
\[
        \EE_X|R_k^{>Q_1}|^2
        =\EE_X|T-M|^2
        \leq 2\EE_XT^2+2M^2=o(1).
\]

It remains to estimate the part with $p^a\leq Q_1$.  Expanding the second moment and applying Lemma \ref{lem:periodic_average} to each pair of residue functions gives an interval-to-complete-residue error at most
\[
        \ll \frac1X\left(\sum_{\substack{p\leq k,\ a\geq2\\p^a\leq Q_1}}(\log p)p^a\right)^2
        +\frac1X\left(\sum_{\substack{p\leq k,\ a\geq2\\p^a\leq Q_1}}\log p\right)^2
        \ll \frac{Q_1^2k^2}{X}+\frac{k^2(\log Q_1)^2}{X}=o(1).
\]
For the complete-residue average, Lemma \ref{lem:crt_factorization} separates the different primes.  For each $p\leq k$, let $b_p$ be the largest integer with $p^{b_p}\leq Q_1$, let $N_p$ be uniform modulo $p^{b_p}$, and set
\[
        H_p:=\sum_{2\leq a\leq b_p}
        \left(\one_{N_p\bmod p^a<[k]_{p^a}}-
        \frac{[k]_{p^a}}{p^a}\right),
\]
with the convention that $H_p=0$ if $b_p<2$.  The complete second moment is
\[
        \sum_{p\leq k}(\log p)^2\EE H_p^2,
\]
because the centered prime-tower variables have mean zero and are independent for distinct primes.

If $p\leq\sqrt k$, there are $O(\log k/\log p)$ levels $a\geq2$ with $p^a\leq k$, and their centered sum is bounded deterministically by this number.  The levels $p^a>k$ are nested as in Lemma \ref{lem:tower_moments}, and their centered contribution has bounded second moment.  Consequently
\[
        \EE H_p^2\ll 1+\left(\frac{\log k}{\log p}\right)^2
        \qquad(p\leq\sqrt k).
\]
If $p>\sqrt k$, then every level $a\geq2$ satisfies $p^a>k$ and hence $[k]_{p^a}=k$.  With
\[
        G_p:=\sum_{2\leq a\leq b_p}\one_{N_p\bmod p^a<k},
\]
the events are nested: $G_p\geq r$ implies $N_p\bmod p^{r+1}<k$.  Therefore
\[
        \PP(G_p\geq r)\leq \frac{k}{p^{r+1}}
        \qquad(r\geq1).
\]
Thus
\[
        \EE G_p^2\ll \sum_{r\geq1}r\,\PP(G_p\geq r)
        \ll \sum_{r\geq1}r\,\frac{k}{p^{r+1}}
        \ll \frac{k}{p^2}.
\]
Since $H_p=G_p-\EE G_p$, the same bound holds for $\EE H_p^2$ up to an absolute constant.

Combining the two prime ranges,
\begin{align*}
\sum_{p\leq k}(\log p)^2\EE H_p^2
&\ll
\sum_{p\leq\sqrt k}(\log p)^2\left(1+\left(\frac{\log k}{\log p}\right)^2\right)
+k\sum_{\sqrt k<p\leq k}\frac{(\log p)^2}{p^2}  \\
&\ll \sqrt k(\log k)^2=o(k\log k).
\end{align*}
Therefore $\EE_X|R_k^{\leq Q_1}|^2=o(k\log k)$, and the tail estimate completes the proof.
\end{proof}

\begin{proof}[Proof of Theorem \ref{thm:clt}]
For all sufficiently large $X$, $k\leq A\log X<X\leq n$ on $\II_X$.  By Lemma \ref{lem:kummer} and the definition of $m(k)$,
\[
        U_k(n)-m(k)=W_k(n)+R_k(n).
\]
Lemma \ref{lem:variance_constant} gives $V(k)\sim \kappa k\log k$, with $\kappa=2-\log(2\pi)>0$.  Lemma \ref{lem:higher_powers} gives
\[
        \frac{R_k(n)}{\sqrt{V(k)}}\to0
        \qquad\text{in }L^2(\PP_X),
\]
and Lemma \ref{lem:prime_clt} gives
\[
        \frac{W_k(n)}{\sqrt{V(k)}}\Rightarrow \normal(0,1).
\]
The first asserted central limit theorem follows from Slutsky's theorem.

It remains to justify the statements about the interval mean and variance.  By the $r=1$ case of the moment transfer in Lemma \ref{lem:prime_clt},
\[
        \EE_X W_k(n)=o(\sqrt{V(k)}).
\]
By Cauchy's inequality and Lemma \ref{lem:higher_powers},
\[
        \EE_X R_k(n)=o(\sqrt{k\log k})=o(\sqrt{V(k)}).
\]
Therefore
\[
        \EE_X U_k(n)=m(k)+o(\sqrt{V(k)}).
\]
For the second moment, the decomposition and Cauchy's inequality imply
\begin{align*}
\EE_X|U_k(n)-m(k)|^2
&=\EE_X W_k(n)^2
  +O\left((\EE_XW_k(n)^2)^{1/2}(\EE_XR_k(n)^2)^{1/2}\right)\\
&\qquad +\EE_XR_k(n)^2  \\
&=V(k)+o(k\log k).
\end{align*}
Here $\EE_XW_k(n)^2=V(k)+o(V(k))$ follows from the $r=2$ case of Lemma \ref{lem:prime_clt}.  Subtracting the square of the mean shift, which is $o(V(k))$, gives
\[
        \Var_X(U_k(n))\sim V(k).
\]
The fully standardized central limit theorem follows by another application of Slutsky's theorem, using both
\[
        \frac{\EE_X U_k(n)-m(k)}{\sqrt{V(k)}}\to0
        \qquad\text{and}\qquad
        \frac{\Var_X(U_k(n))}{V(k)}\to1.
\]
\end{proof}

\section{Concluding remarks}\label{sec:remarks}

Theorem \ref{thm:main} describes the normal order of the first crossing for the general threshold $u(n,k)>n^c$.  Its density-one formulation is essential: the proof deliberately discards a zero-density set of integers for which large prime powers with small bases divide one of the nearby integers
\[
        n,n-1,\ldots,n-\lfloor A\log X\rfloor.
\]
Such congruence obstructions are invisible at natural density one but can matter in worst-case questions.  This is why the result should not be quoted as a resolution of the original pointwise problem.  What it supplies is the typical location of the first crossing, including the exact leading constant.

The constant $1-\gamma$ is the key arithmetic feature in the normal-order theorem.  A first-order heuristic might suggest that the small-prime part crosses $n^c$ when $k\approx c\log n$.  The complete-residue average instead equals
\[
        k\sum_{p\leq k}\frac{\log p}{p-1}-\log k!
        =(1-\gamma)k+o(k),
\]
which moves the normal-order threshold to
\[
        \frac{c}{1-\gamma}\log n.
\]
For the original threshold $u(n,k)>n^2$, this gives the special case
\[
        \frac{2}{1-\gamma}\log n=4.730544237\ldots\log n
\]
for almost all $n$.

The fluctuation theorem has a different constant.  On the scale $\sqrt{k\log k}$, the higher prime-power levels are negligible and the variance comes from the primes $p\leq k$ alone.  This yields
\[
        2-\log(2\pi)
        =\int_1^\infty \frac{\fracp{y}(1-\fracp{y})}{y^2}\,dy.
\]

\section*{Acknowledgements}
The author acknowledges the use of OpenAI's ChatGPT during the preparation of this manuscript. While it was used for ideation, formulation, proof exploration and refinement, narrowing the search space, programming, LaTeX formatting and other forms of orchestration, the author nonetheless takes full responsibility for the accuracy of the final contents of this paper.

\end{document}